\documentclass[a4paper,11pt]{article}

\usepackage{amsmath,amsthm}
\usepackage{graphicx}
\usepackage[margin = 1 in]{geometry}

\theoremstyle{definition}
\newtheorem{definition}{Definition}
\newtheorem{term}{Term}

\theoremstyle{plain}
\newtheorem{theorem}{Theorem}

\newtheorem{corollary}{Corollary}
\newtheorem{claim}{Claim}
\usepackage{sectsty}
\usepackage{bm}
\usepackage[T1]{fontenc}
\title{Domination broadcast: A case study on a combination of cycle graph and sunlet graph}
\author{
\\ \\ 
Sivakorn Sanguanmoo \\ 
Department of Mathematics \\
University of Wisconsin-Madison\\
{\tt sanguanmoo@wisc.edu}\\ \\
}
\date{}
\begin{document}
\maketitle
\begin{abstract}
Domination in graphs has long been studied and is applied to signal distribution problem. For example, telecommunication companies want to spread the signal from broadcast stations by transmitting varying signal strength to all receiving stations. This problem can be interpreted in a term of graph theory. Assume that broadcast companies need to spread the signal in graphs by using broadcast stations with varying signal strength so that the signal could be sent to all the vertices. However, broadcast stations with stronger signal are generally more expensive.  The distribution of signal was configured to cover all stations with minimum total cost of signal called the $\gamma_b-$dominating broadcast number. This paper shows another proof of the $\gamma_b-$dominating broadcast number of cycle graphs and sunlet graphs as a foundation for the further result. I also consider the $\gamma_b-$dominating broadcast number of a generalized version of sunlet graphs whose vertices on the cycle are equally extended by the path, called the sunlet graph with degree $n$. To obtain the optimal cost of the signal distribution for this extended version, we show that it is sufficient to use only one broadcast station at a vertex on the cycle with the signal cost equal to the radius of the sunlet graph with degree $n$.

\end{abstract}
\newpage
\section{Introduction}

Daily life network communication has been found problematic. Distribution network is a part of network communication. However, the distribution of signal is not covering the target areas. Also, the strong signal distributed to cover the target areas is quite expensive. In this study, we examined relation of network by means of graphs. We examined how to construct signal stations that could distribute signal to cover the target areas, where each of which has determined distribution strength with minimum total cost of signal sending. We considered all of the areas as a set of vertices, where two vertices were joined with an edge when the two areas were neighbor. All of vertices and edges were called A graph $G$. All of the signal stations that could satisfy the previous conditions were represented by broadcasts dominating set $S$ so we called the cardinality of set $S$ as $\gamma_b-$ broadcast number of graph $G$ or $\gamma_b(G)$.  
     
     Some specific graphs of this study represented cities in general. For example, a sunlet graph is a graph representing the center and each center connects to another which is like an urban city. Also, pendant vertices represent cities in the countryside. Therefore, studying a function $\gamma_b$ of specific graphs, such as cycle graph $(C_n)$, sunlet graph $(S_n)$, and sunlet$-n—$, is going to be useful for broadcasting signal in the real world.    
     
\section{Term in Graph Theory}  

\begin{term} The \textbf{distance} $d(u,v)$  between vertices $u$ and $v$ of a graph $G$ is the length of a shortest path between $u$ and $v$ \end{term}

\textbf{Example}

\begin{figure} [htbp]
\centering
\includegraphics[width=7cm]{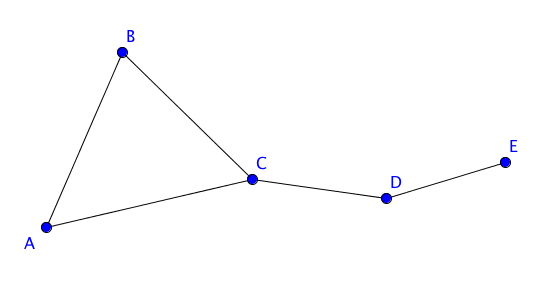}
\caption{The example of graph $G$}
\end{figure}

According to \textbf{Term 1}, $d(A,B) =1, d(A,D) =2$ and $d(B,E) =3$

\begin{term} \textbf{Eccentricity of a vertex v or $e(v)$} is the maximum distance of d(u,v) among all vertices $u$ in a graph G. Formally, $e(v) = max\{d(u,v)|u \in V(G)\}$. \end{term}

\begin{term}\textbf{Radius of a graph $G$ or $rad(G)$} is the minimum eccentricity value of a vertex $v$ among all vertices $v$ in a graph $G$. Formally, $rad(G) = min\{e(v)|v \in V(G)\}$.\end{term}

\begin{term}\textbf{Diameter of a graph $G$ or $diam(G)$} is maximum eccentricity value of a vertex $v$ among all vertices $v$ in a graph $G$. Formally, $diam(G) = max\{e(v)|v \in V(G)\}$. \end{term}

From \textbf{Figure 1}, $e(A)=max\{d(A,B),d(A,C),d(A,D),d(A,E)\}=3$. Similarly, $e(B)=3,e(C)=2,e(D)=2$, and $e(E)=3$, so $rad(G)=2$ and $diam(G)=3$ 

\section{Definition}

\subsection{Variations graph of cycle graph and sunlet graph}

\begin{definition} \textbf{The path} or \textbf{$P_n$} is a graph which consists of $n$ vertices lying on a straight line where $n$ is a natural number. Formally, $V(P_n)=\{v_1,v_2,...,v_n\},E(P_n)=\{v_iv_{i+1}|i\in \{1,2,...,n-1\}\}$ \end{definition}

\begin{figure} [htbp]
\centering
\includegraphics[width=8cm]{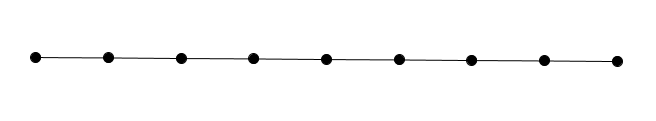}
\caption{graph $P_9$}
\end{figure}

\begin{definition} \textbf{The cycle} or \textbf{$C_n$} is a graph which consists of $n$ vertices, $v_1,v_2,...,v_n$, and edges, $v_1v_2, v_2v_3,...,v_{n-1}v_n,v_nv_1$ \end{definition}

\begin{figure} [htbp]
\centering
\includegraphics[width=8cm]{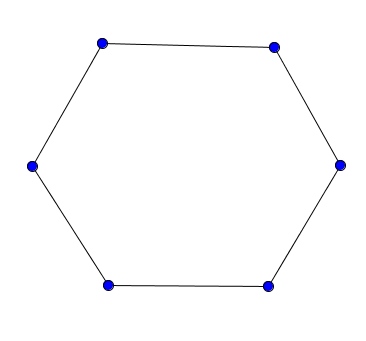}
\caption{graph $C_6$}
\end{figure}

\begin{definition} \textbf{The sunlet} or \textbf{$S_n$} is a graph on vertices obtained by attaching pendant edges to a cycle graph $C_n$. \end{definition}

\begin{figure} [htbp]
\centering
\includegraphics[width=7cm]{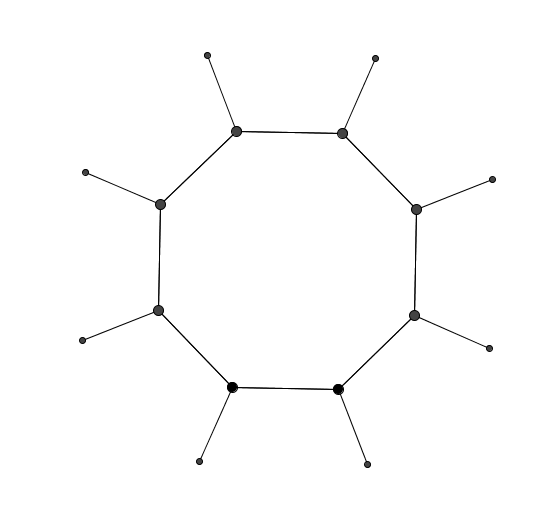}
\caption{graph $S_8$}
\end{figure}

\begin{definition}[\cite{anitha2008n}] \textbf{$m-$ Sunlet degree \textbf{$n$} or $S_m^n$} is a graph which consists of $C_m$ and straight path length $n$  $(P_n)$ extended from all vertices of $C_m$ of where $m,n$ are positive integers such that $m\geq 3, n \geq 1$. We call each vertex of $C_m$ as a \textbf{base vertex}, each $P_n$ as a \textbf{branch}, vertices on $P_n$ as \textbf{pendant vertices}, and an outmost pendant vertex as a \textbf{leaf vertex}.  \end{definition}  

\begin{figure} [htbp]
\centering
\includegraphics[width=6cm]{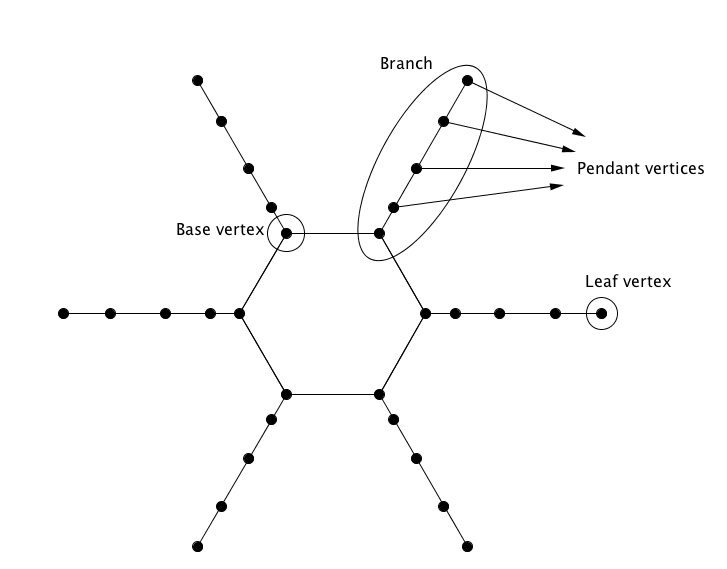}
\caption{graph $S_6^4$}
\end{figure}

\subsection{Dominating Broadcast \cite{haynes1998fundamentals},\cite{haynes1998domination}}

\begin{definition} \textbf{A broadcast vertex $v$} with \textbf{strength} $m$ is a vertex which can send a \textbf{signal} to vertices in the graph such that a vertex $u$ can receive a signal from $v$ if and only if $d(u,v)\leq m$, note that $u$ can send a signal to itself where $m$ is a positive integer. \end{definition}

\begin{figure} [htbp]
\centering
\includegraphics[width=10cm]{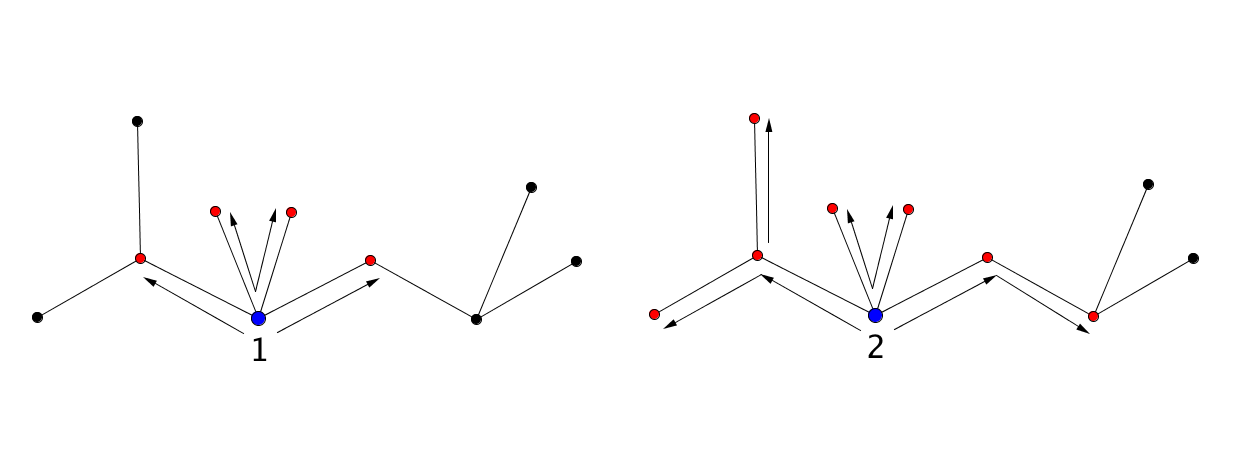}
\caption{graph H with a broadcast vertex (the blue vertex) with signal strength (the number beside the blue vertex) and vertices which can receive a signal from the broadcast vertex (red vertices)}
\end{figure}

\begin{definition} \textbf{Dominating broadcast function $f$} is a function from set of vertices to set of positive integers including $0$ such that if a function $f$ value of a vertex $u$ is $0$, $u$ is not a broadcast vertex. On the other hand, if a function $f$ value of the vertex $u$ is $m \ne 0$, $u$ is a broadcast vertex with strength $m$. The condition of this function is that all vertices of the graph have to receive signals from at least one broadcast vertex. 
\end{definition}

\begin{definition} \textbf{The cost of function $f$} is the total of signal strengths of broadcast vertices in function $f$.
\end{definition}

\begin{figure} [htbp]
\centering
\includegraphics[width=8cm]{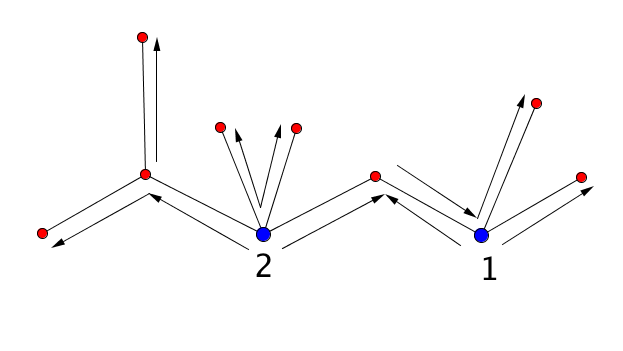}
\caption{graph $H$ with the dominating broadcast function (two blue vertices are broadcast vertices and a number beside each broadcast vertex is a signal strength of that vertex). }
\end{figure}

\begin{definition}
\textbf{$\gamma_b-$ dominating broadcast function of graph $G$} is a dominating broadcast function of $G$ which has the minimum cost. We call this minimum cost as\textbf{ $\gamma_b$ of a graph $G$ }or \textbf{$\gamma_b(G)$}
\end{definition}

As we see in \textbf{Figure 7}, those two vertices with those signal strengths have the minimum cost of dominating broadcast functions, so $\gamma_b(H)=3$. 

\begin{definition} \textbf{The efficient broadcast} is the dominating broadcast such that each vertex can receive a signal from only one broadcast vertex.
\end{definition}

As we see in \textbf{Figure 7}, this is not the efficient broadcast because the vertex between two broadcast vertices can receive signals from both broadcast vertices.

\section{Corollary}
These following corollaries are essential corollaries for finding $\gamma_b-$dominating broadcast function value of the cycle, the Sunlet, and the Sunlet degree $n$.

\begin{corollary}[\cite{erwin2001cost}] Every graph $G$ has a $\gamma_b-$ broadcast which is efficient.\end{corollary}

This corollary is proved by D.J. Erwin. We will not show the proof but the example for this corollary.

\begin{figure} [htbp]
\centering
\includegraphics[width=6.5cm]{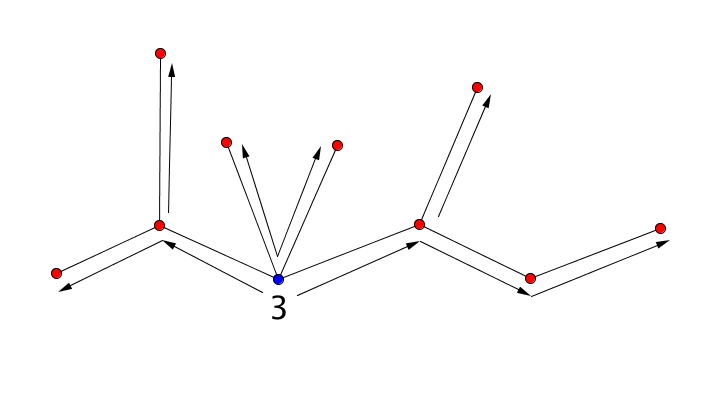}
\caption{The $\gamma_b-$ dominating broadcast function of the graph $H$  which is efficient}
\end{figure}

From \textbf{Figure 7}, $\gamma_b(H)=3$. Therefore, the dominating broadcast function in \textbf{Figure 8} is the $\gamma_b-$ dominating broadcast function. Consider that every vertex can receive a signal from only the blue vertex. Therefore, this $\gamma_b-$ dominating broadcast function is efficient.

Consider \textbf{Figure 8} that the efficient broadcast function which is also the $\gamma_b-$ dominating broadcast function uses only one broadcast vertex. However, it does not mean that every graph can use only one broadcast vertex for the $\gamma_b-$ dominating broadcast function.  

\begin{figure} [htbp]
\centering
\includegraphics[width=6.5cm]{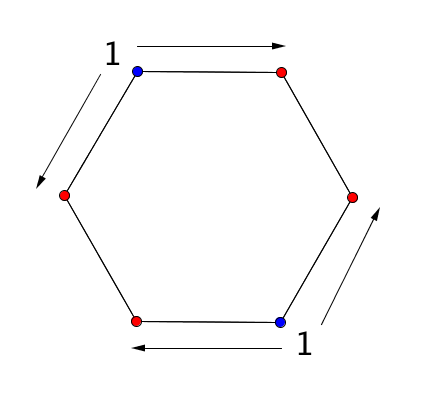}
\caption{The $\gamma_b-$ dominating broadcast function of the graph $C_6$  which is efficient}
\end{figure}

For $C_6$, $\gamma_b(C_6)=2$. If we let a signal strength to any vertex for $2$, it cannot send a signal to all vertices in a graph. In this case, the $\gamma_b-$ dominating broadcast function occurs by using at least two vertices. The $\gamma_b-$ dominating broadcast function in \textbf{Figure 9} is efficient.

\begin{corollary} $rad(S_m^n)=\lfloor \frac{m}{2}\rfloor+n$\end{corollary}

\begin{proof} From \textbf{Term 2}, $rad(S_m^n) = min\{e(v)|v \in V(G)\}$. We need to find a vertex $v$ which has the minimum eccentricity value. We will prove that $v$ needs to be the base vertex. For the sake of contradiction, assume that $v$ is the pendant vertex with its base vertex $u$. Consider the vertex $w$ which is not on the branch consisting the vertex $v$. It is obvious that the shortest path between $v$ and $w$ has to pass $u$. Therefore,
\begin{align}
d(v,w)=d(v,u)+d(u,w)> d(u,w) \nonumber
\end{align}

\begin{figure} [htbp]
\centering
\includegraphics[width=6.5cm]{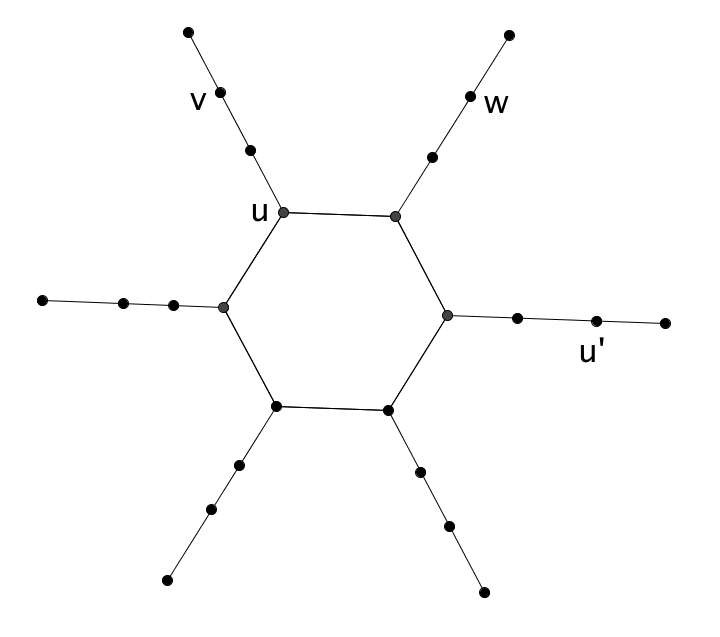}
\caption{The graph $S_6^3$ with vertices $u,v,w,$ and $u'$}
\end{figure}

Let $e(u)=d(u,u')$ where $u'$ is the vertex of the graph $G$. Consider that $u'$ and $u$ cannot be on the same branch; otherwise, $d(u,u')\leq n$, which is not true because the distance between $u$ and a leaf vertex on the different branch is more than $n$. Therefore,   
\begin{align}
e(u)=d(u,u')<d(v,u')\leq e(v) \nonumber
\end{align}
, which contradicts with the minimality of $e(v)$ as desired.

We can conclude that $v$ has to be a base vertex, and the farthest vertex from $v$ is the leaf vertex $t$ on the branch completely opposite to the vertex $v$. Therefore, 
\begin{align}
rad(S_m^n)=e(v)=d(v,t)=\lfloor \frac{m}{2}\rfloor+n \nonumber
\end{align} \end{proof}
\begin{figure} [htbp]
\centering
\includegraphics[width=6.5cm]{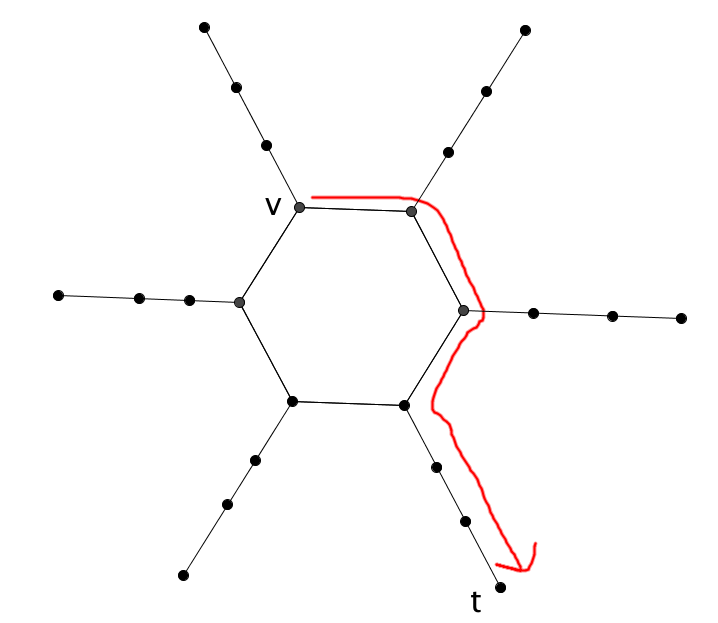}
\caption{The shortest path between $v$ and $t$ of the graph $S_6^3$}
\end{figure}

\begin{corollary} For every graph $G$, $\gamma_b(G) \leq rad(G) $ \end{corollary}

\begin{proof} From \textbf{Term 2}, $rad(G) = min\{e(v)|v \in V(G)\}$. Therefore, there is a vertex $v'$ such that $e(v')=rad(G)$, which means that $d(v',u) \leq rad(G)$ for every vertex $v$. We can use $v'$ as a broadcast vertex with a strength $rad(G)$ to send a signal to all vertices in the graph. Thus, $\gamma_b(G)\leq rad(G)$ as desired.
\end{proof}

\begin{corollary}[\cite{dunbar2006broadcasts}] For every natural number $n$, $\gamma_b(P_n) = \lceil \frac{n}{3} \rceil $ \end{corollary}

We will not show the proof, but we will introduce \textbf{the standard pattern of ${\gamma}_{b}-$dominating broadcast function of $P_n$}

\begin{definition} \textbf{the standard pattern of ${\gamma}_{b}-$dominating broadcast functions of $P_n$} is the ${\gamma}_{b}-$dominating broadcast function of $P_n$ with vertices $v_1,v_2,...,v_n$ with following conditions
\begin{itemize}
\item if $n=1$, then let $P_1$ be a broadcast vertex with a strength $1$.
\item if $n=2$, then let either $P_1$ or $P_2$ be a broadcast vertex a strength $1$.
\item if $n=3k$ where $k$ is a natural number, then let $P
_2,P_5,...,P_{3k-1}$ be broadcast vertices with a strength $1$
\item if $n=3k+1$ where $k$ is a natural number, then let either $P_1,P_3,P_6,...,P_{3k}$ or $P_2,P_4,P_7,...,P_{3k+1}$ be broadcast vertices with a strength $1$.
\item if $n=3k+2$ where $k$ is a natural number, then let $P_2$ be a broadcast vertex with a strength $2$, and let $P_7,P_{10},...,P_{3k+1}$ be broadcast vertices with a strength $1$.
\end{itemize}
\end{definition}

It is easy to see that the cost of those standard patterns function is $\lceil \frac{n}{3} \rceil$. Thus, these functions are the ${\gamma}_{b}-$dominating broadcast function of $P_n$.

\begin{figure} [htbp]
\centering
\includegraphics[width=15cm]{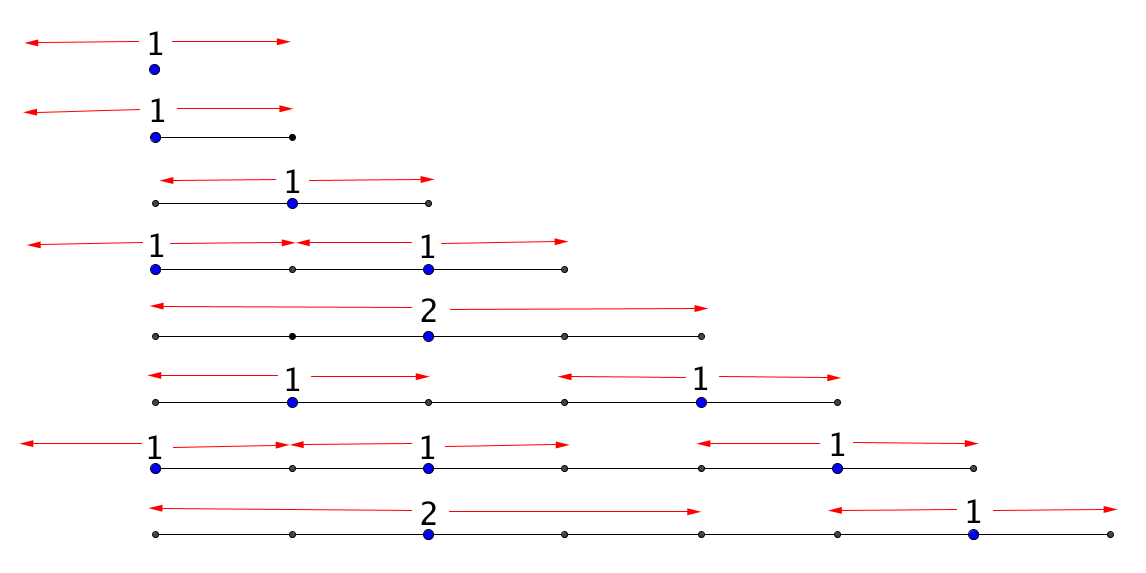}
\caption{the standard pattern of ${\gamma}_{b}-$dominating broadcast function of $P_1,P_2,...,P_8$}
\end{figure}

\textbf{The significant property} of the standard pattern of ${\gamma}_{b}-$dominating broadcast functions of $P_n$ is that if we add "phantom vertices" on both sides of $P_n$ and "phantom edges" to be the infinite path, those broadcast vertices can send a signal to "phantom vertices" at most one vertex for each side.

\begin{corollary} Let $G$ be a graph consisting of vertices $A$ and $B$, but $A$ is not adjacent to $B$. Let $G'$ be the same graph with $G$ but an external side $AB$. Then, ${\gamma}_{b}(G')\leq{\gamma}_{b}(G)$.
\end{corollary}

\begin{proof} This corollary is obvious because the ${\gamma}_{b}-$dominating broadcast function of $G$ is a dominating broadcast function of $G'$ because adding edges does not interrupt signal sending.
\end{proof}
\section{Main theorem}
The main purpose of this research is to find $\gamma_b-$dominating broadcast function value of $C_n$ and $S^n_m$. 
\begin{theorem} 
\label{1}
$\gamma_b(C_n)=\lceil \frac{n}{3}\rceil$ for each natural number $n\geq 3$ \end{theorem}

\begin{figure} [htbp]
\centering
\includegraphics[width=6cm]{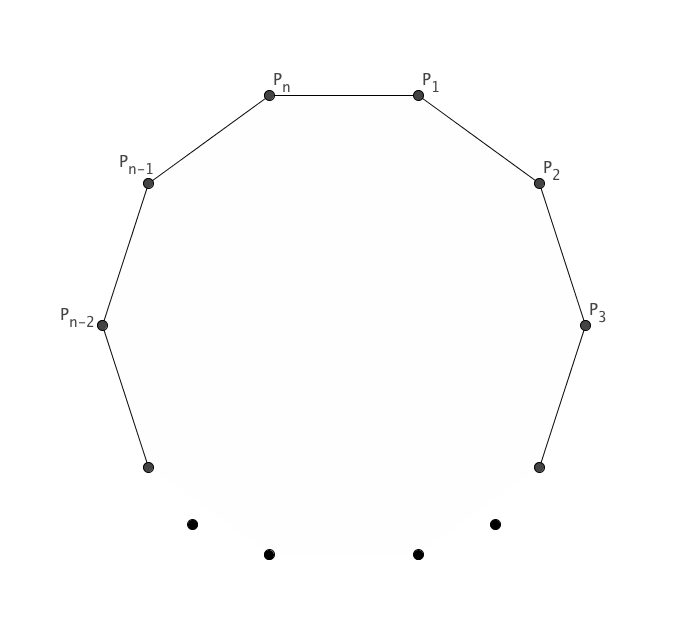}
\caption{the cycle graph $C_n$}
\end{figure}

\begin{proof}
Consider any dominating broadcast function $x$ of $C_n$. Let broadcast vertices for the function $x$ are $P_{x_1},P_{x_2},...,P_{x_k}$ where $k$ is the number of the broadcast vertices, and $x_1,x_2,...,x_k$ are natural number in an ascending order. Let the strength of each $P_{x_l}$ be $q_l$ for each natural number $l \in \{1,2,..,k\}$.

\begin{figure} [htbp]
\centering
\includegraphics[width=8cm]{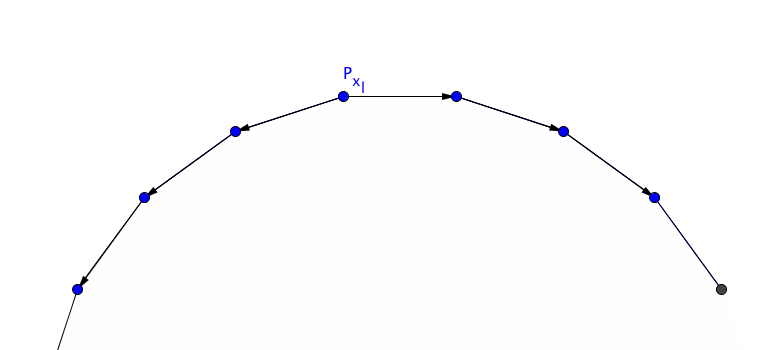}
\caption{Signals from $P_{x_l}$ (blue vectors)}
\end{figure}

Consider that $P_{x_l}$ can send a signal to its clockwise direction for $q_l$ vertices. Also, it can send a signal to its anticlockwise direction for $q_l$ vertices. Therefore, $P_{x_l}$ can send a signal to $q_l+q_l+1=2q_l+1$ vertices (including $P_{x_l}$ itself). 

Because $x$ is a dominating broadcast function, every vertex can receive signal at least once. Therefore,

\begin{align}
\sum_{l=1}^k {(2q_l+1)}  \geq n \nonumber\\
2\sum_{l=1}^k {q_l} + k \geq n \nonumber\\
\sum_{l=1}^k {q_l} \geq \frac{n-k}{2} \nonumber
\end{align}
Also, $q_l \geq 1$ for each natural number $k \in \{1,2,..,k\}$, so $\sum_{l=1}^k {q_l} \geq k$. Thus,
\begin{align}
\sum_{l=1}^k {q_l} \geq max\{\frac{n-k}{2},k\} \nonumber
\end{align}
Consider $\frac{n-k}{2}+\frac{n-k}{2}+k=n$, from Pigeonhole's principle, $max\{\frac{n-k}{2},k\}\geq \lceil\frac{n}{3}\rceil$
Therefore, 
\begin{align}
\sum_{l=1}^k {q_l} \geq \lceil\frac{n}{3}\rceil \nonumber
\end{align}
which means that the minimum cost of $C_n$ is at least $\lceil\frac{n}{3}\rceil$. It is sufficient to find the example of the $\gamma_b-$dominating broadcast function $x$ of $C_n$ with cost $\lceil\frac{n}{3}\rceil$. We provide $3$ following cases.

\textbf{Case $1$} $n = 3k$ for some positive integer $k$

Let $P_3,P_6,...,P_{3k}$ are broadcast vertices such that signal strength of each broadcast vertex is $1$. Therefore, the cost of this function is $k=\lceil{\frac{n}{3}}\rceil$. Consider that $P_{3i+1}$ can receive a signal from $P_{3i}$, and $P_{3i+2}$ can receive a signal from $P_{3i+3}$. Thus, this dominating broadcast function can send a signal to cover all vertices in the graph as desired.

\begin{figure} [htbp]
\centering
\includegraphics[width=8cm]{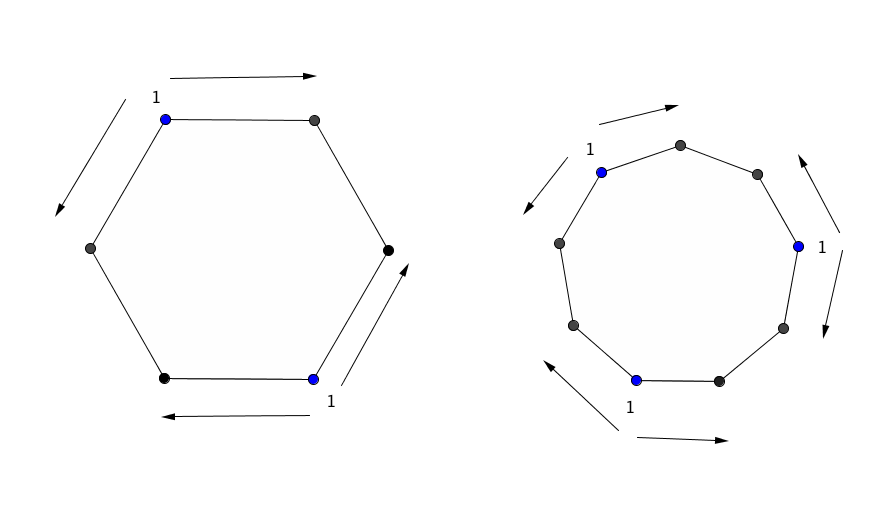}
\caption{an example of the $\gamma_b-$dominating function of $C_3$ (left) and $C_6$ (right) }
\end{figure}

\textbf{Case $2$} $n=3k+1$ for some positive integer $k$

Let $P_3,P_6,...,P_{3k}$ and $P_1$ are broadcast vertices such that signal strength of each broadcast vertex is $1$. Therefore, the cost of this function is $k+1 =\lceil{\frac{n}{3}}\rceil$. Consider that $P_{3i+1}$ can receive a signal from $P_{3i}$, and $P_{3i+2}$ can receive a signal from $P_{3i+3}$. Thus, this dominating broadcast function can send a signal to cover all vertices in the graph as desired.

\begin{figure} [htbp]
\centering
\includegraphics[width=8cm]{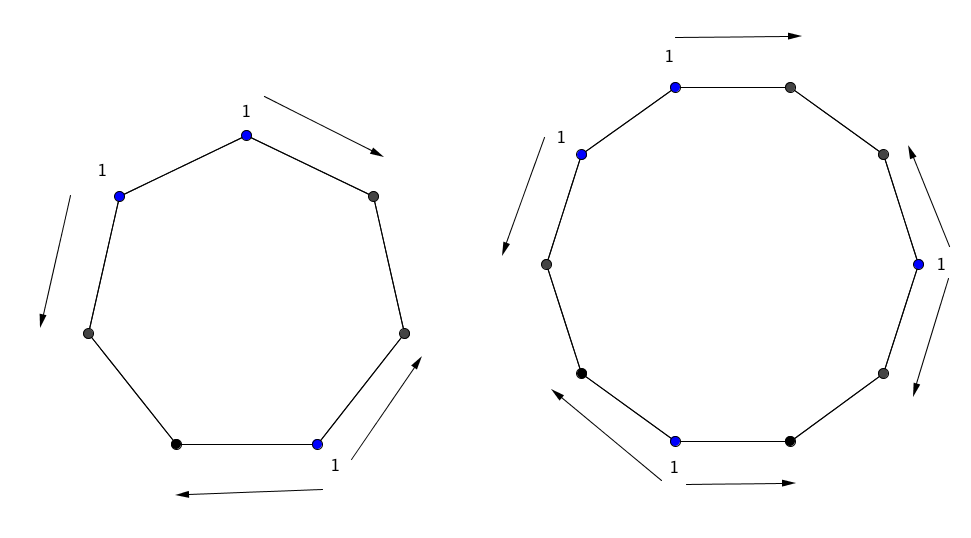}
\caption{an example of the $\gamma_b-$dominating function of $C_7$ (left) and $C_{10}$ (right) }
\end{figure}

\textbf{Case $3$} $n=3k+2$ for some positive integer $k$

Let $P_3,P_6,...,P_{3k}$ and $P_1$ are broadcast vertices such that signal strength of each broadcast vertex is $1$. Therefore, the cost of this function is $k+1 =\lceil{\frac{n}{3}}\rceil$. Consider that $P_{3i+1}$ can receive a signal from $P_{3i}$, and $P_{3i+2}$ can receive a signal from $P_{3i+3}$. Finally, $P_{3k+2}$ can receive a signal from $P_{1}$. Thus, this dominating broadcast function can send a signal to cover all vertices in the graph as desired.

\begin{figure} [htbp]
\centering
\includegraphics[width=8cm]{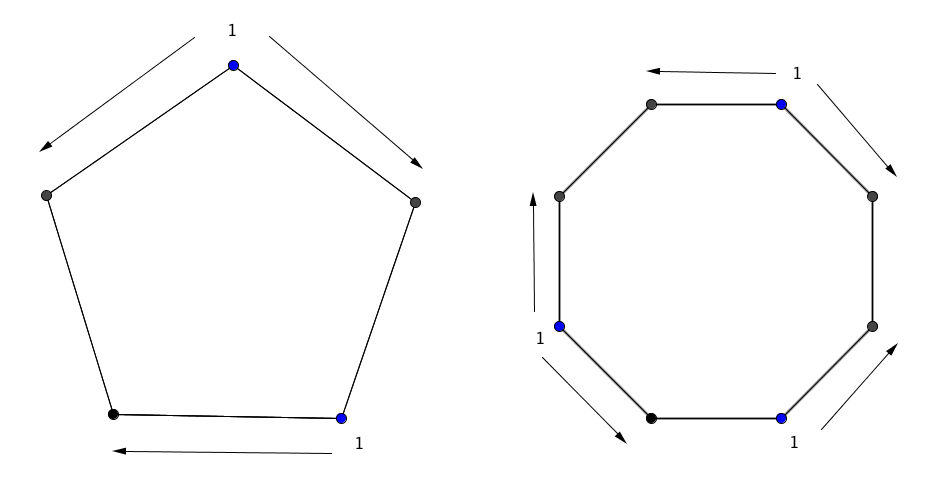}
\caption{an example of the $\gamma_b-$dominating function of $C_5$ (left) and $C_{8}$ (right) }
\end{figure}

From all three cases, we can conclude that $\gamma_b(C_n)=\lceil \frac{n}{3}\rceil$ for each natural number $n\geq 3$ \end{proof}

\begin{theorem}  
\label{2}
$\gamma_b(S_n)=\lceil \frac{n+1}{2}\rceil$ for each natural number $n\geq 3$ \end{theorem}

\begin{figure} [htbp]
\centering
\includegraphics[width=8cm]{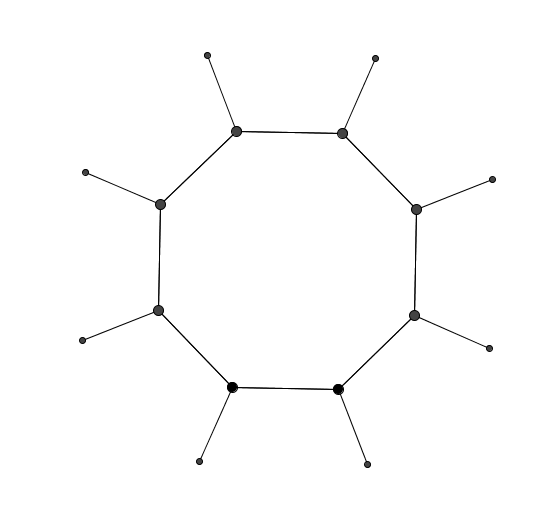}
\caption{Sunlet graph $S_8$}
\end{figure}

\begin{proof} 
Consider any $\gamma_b-$dominating broadcast $x$, which is also efficient broadcast.

\begin{claim} there is only one broadcast vertex for the efficent broadcast $x$. \end{claim}

\begin{proof} For the sake of contradiction, assume that there are at least two broadcast vertices for the function $x$. Because the number of branches is finite, there are two different broadcast vertices, $v_1,v_2$, which locate on two branches such that there is no another broadcast vertex on the short arc between two branches. 

It is obvious that $v_1,v_2$ are on different branches; otherwise, $v_1$ can receive signals from both $v_1$ and $v_2$, which contradicts with the property of the efficient broadcast function.

\begin{figure} [htbp]
\centering
\includegraphics[width=8cm]{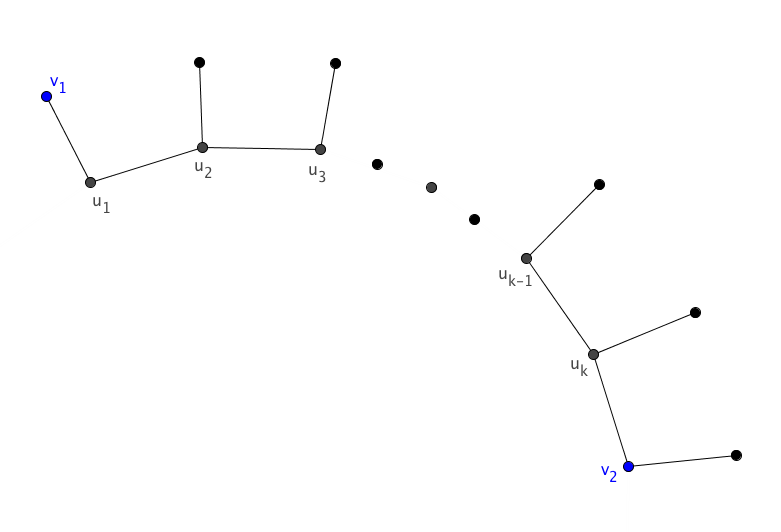}
\caption{The shortest path between $v_1$ and $v_2$ (blue vertices)}
\end{figure}

Let vertices on the shortest path between $v_1$ and $v_2$ be $v_1,u_1,u_2,...,u_k,$ and $v_2$ in this order. Consider that $u_1,u_2,...,u_k$ cannot be pendant vertices because the degree of pendant vertices is just $1$, so it is impossible to have a pendant vertex as a vertex between the path.  

By the definition, $u_1,u_2,...,u_k$ cannot receive from broadcast vertices, except $v_1,v_2$. Let $i \in \{1,2,...,k\}$ be the largest positive integer such that $v_1$ can send a signal to $u_i$. Therefore, $u_{i+1}$ cannot receive a signal from $v_1$, so the signal strength of $v_1$ is $i$.

However, the distance between $v_1$ and the pendant vertex $u'_i$ of the branch which consists of $u_i$ is $i+1$. Thus, $u'_i$ cannot receive a signal from $v_1$. Therefore, $u'_i$ has to receive a signal from $v_2$. Consider that the only shortest path between $u'_i$ and $v_2$ consists of $u_i$, which means that $u_i$ can receive a signal from $v_2$. We can conclude that $u_i$ can receive signals from both $v_1$ and $v_2$, which contradicts with the definition of the efficient broadcast, as desired. \end{proof}

We know that the minimum cost of $\gamma_b-$ dominating broadcast occurs when there is just one broadcast vertex. Thus, $rad(S_n)= \gamma_b(S_n)$. Consider that $rad(S_n)=\lceil \frac{n+1}{2} \rceil$, therefore, $\gamma_b(S_n)=\lceil \frac{n+1}{2} \rceil.$\end{proof}.

\begin{figure} [htbp]
\centering
\includegraphics[width=8cm]{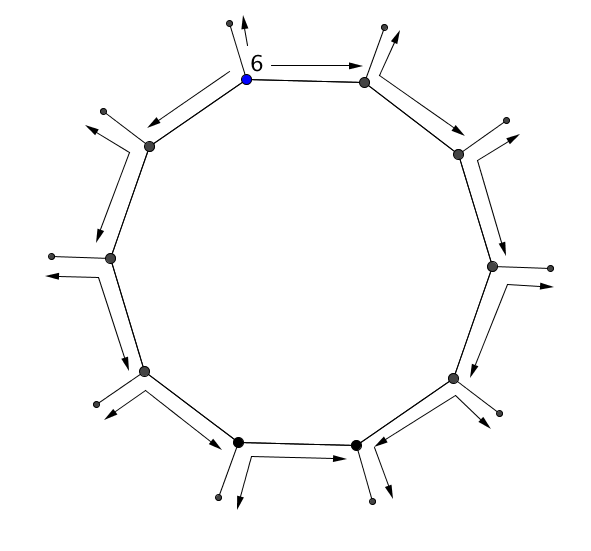}
\caption{The example of $\gamma_b-$ dominating broadcast of $S_{10}$}
\end{figure}

Next, we will consider Sunlet degree $n$ where its cycle graph is $C_3$
\begin{theorem} 
\label{3}
$\gamma_b(S^m_3)=m+1$ for each natural number $m$ \end{theorem}

\begin{proof} Consider the $\gamma_b-$domianting broadcast $x$ which is also an efficient broadcast. Let $X$ is a broadcast vertex with strength $k$, which can send a signal to any base vertex. We will provide $2$ cases of $X$.

\textbf{Case 1} $X$ can send a signal to exactly one base vertex.

Therefore, $X$ cannot be a base vertex, and the distance between $X$ and its base vertex is exactly $k$. Then, the remaining vertices will be separated into two paths as \textbf{Figure 21} shown.

\begin{figure} [htbp]
\centering
\includegraphics[width=8cm]{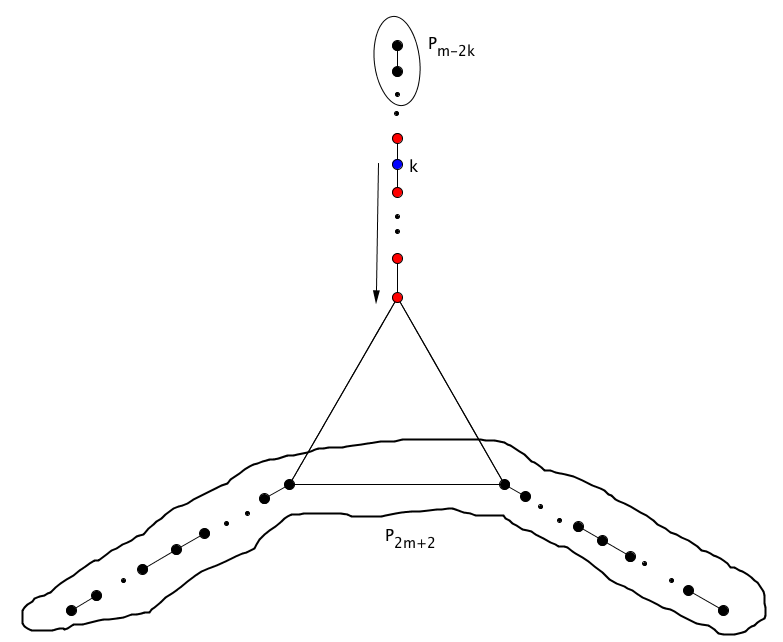}
\caption{Two disconnected paths (in closed loops) in a case that $X$ can send a signal to exactly one broadcast vertex}
\end{figure}

Therefore, the cost of this dominating function is $k+\lceil\frac{m-k}{3}\rceil+ \lceil\frac{2m+2}{3}\rceil$

\begin{align}
k+\lceil\frac{m-2k}{3}\rceil+ \lceil\frac{2m+2}{3}\rceil &\geq \lceil k+\frac{m-2k}{3}+\frac{2m+2}{3}\rceil\nonumber\\
&= \lceil \frac{3m+k+2}{3}\rceil\nonumber\\
&\geq \lceil \frac{3m+3}{3}\rceil\nonumber\\ 
&= m+1=rad(S_3^m) \nonumber 
\end{align}

\textbf{Case 2} $X$ can send a signal to at least two broadcast vertices.

Let the distance between $X$ and its base vertex is $l$. By symmetry, $X$ can send a signal to all three broadcast vertices. Then, the remaining vertices which cannot receive a signal from $X$ will be separated into three paths as \textbf{Figure 22} shown.

\begin{figure} [htbp]
\centering
\includegraphics[width=7cm]{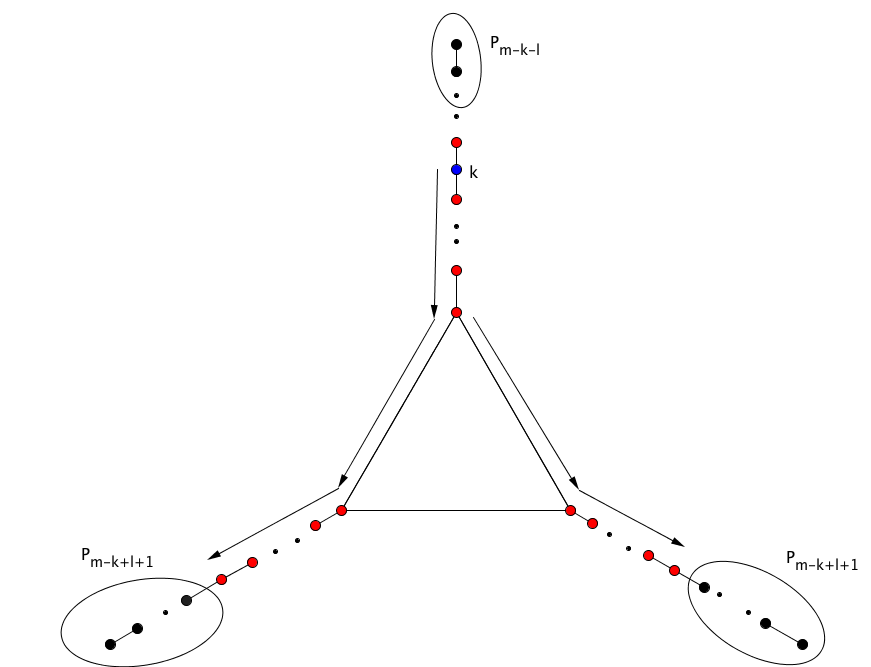}
\caption{Three disconnected paths (in ovals) in a case that $X$ can send a signal to at least two broadcast vertices}
\end{figure}
Therefore, the cost of this dominating function is $k+\lceil\frac{m-k-l}{3}\rceil+ 2\lceil\frac{m-k+l+1}{3}\rceil$

\begin{align}
k+\lceil\frac{m-k-l}{3}\rceil+ 2\lceil\frac{m-k+l+1}{3}\rceil &\geq \lceil k+\frac{m-k-1}{3}+\frac{2(m-k+l+1)}{3}\rceil\nonumber\\
&= \lceil \frac{3m+2l+1}{3}\rceil\nonumber\\
&\geq \lceil \frac{3m+1}{3}\rceil\nonumber\\ 
&= m+1=rad(S_3^m) \nonumber 
\end{align}

Therefore, the minimum cost of dominating broadcast function of $S^m_3$ is $rad(S_3^m)$. Thus, 
\begin{align}
\gamma_b(S^m_3)=m+1 \nonumber 
\end{align} \end{proof}

Next, we will use the idea of \textbf{Theorem 2} and \textbf{Theorem 3} for finding the $\gamma_b-$dominating broadcast function value for a general case of Sunlet degree $n$.  
\begin{theorem}  
\label{4}
$\gamma_b(S^n_m)=n+\lfloor \frac{m}{2}\rfloor$ for each natural number $n\geq 1$ and $m \geq 3$
\end{theorem}
\begin{figure} [htbp]
\centering
\includegraphics[width=8cm]{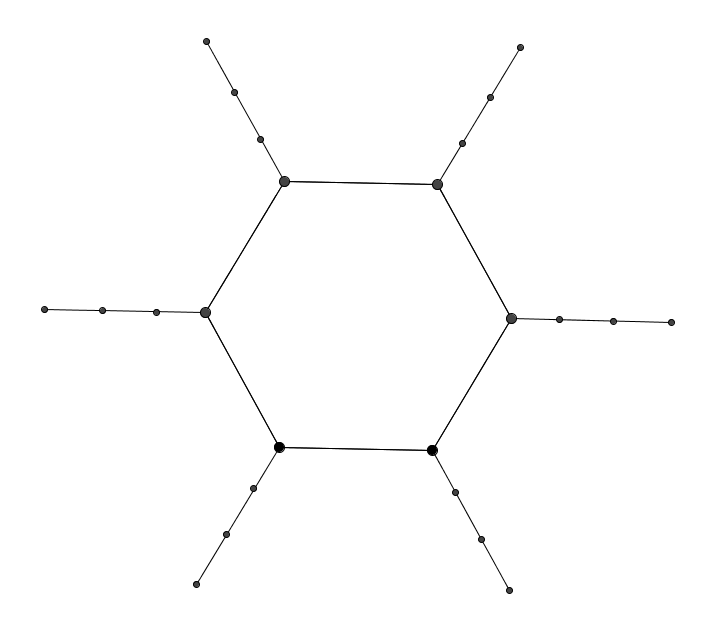}
\caption{Graph $S^6_3$}
\end{figure}

\begin{proof}
Let $P(i)$ represents the statement "$\gamma_b(S^n_i)=i+\lfloor \frac{n}{2}\rfloor$ for each natural number $n \geq 1$," for each natural number $i\geq 3$.
We will use the mathematical induction on a natural number $n$.

For the base case, we need to prove $P(3)$ is true. By \textbf{Theorem 4}, 
\begin{align}
\gamma_b(S^n_3)=n+1=n+\lfloor \frac{3}{2}\rfloor \nonumber
\end{align}

For the induction case, we assume that $P(m-1)$ is true where $m$ is a positive integer greater than or equal 4, and we need to prove that $P(m)$ is true.

Let $A$ be a set of ${\gamma}_{b}-$ broadcast functions of $S^n_m$. Let $A'$ be a subset of the set $A$ such that for each ${\gamma}_{b}$- efficient broadcast function of $S^n_m$, all vertices that can send signal to a base vertex are base vertices. Let $B \subseteq A$ such that each ${\gamma}_{b}-$ broadcast function in $B$ does not has base vertex which receives a signal from two broadcast vertices. It is obvious that $B \ne \emptyset$ because the efficient broadcast function is an element of $B$   

\begin{claim} If $n>2$, there is no function in $B$ such that there is a broadcast vertex $M$ which can send a signal to the nearest base vertex $N$, and its signal strength is $n=d(M,N)$ \end{claim}

\begin{proof} We will prove this claim by the sake of contradiction. Assume that there is a broadcast vertex $M'$ which can send a signal to the nearest base vertex $N'$, and its signal strength is $n'=d(M',N')$. Let the adjacent base vertices of $N'$ be $N'_1,N'_2$.  

Therefore, this broadcast vertex cannot send a signal to a base vertex other than $N'$. Also, there is no signal crossing over the branch consisting of $M',N'$ because otherwise the signal can reach $N'$, which contradicts with the property of the set $B$. Therefore, edges $N'N'_1,N'N'_2$ are not be used for a passing way for signals. Thus, we can cut these two edges, and the $\gamma_b-$ dominating broadcast function value is still the same.

Consider that the new graph is two independent graphs: $P_{n+1}$ and the remaining graph called $S'$. By \textbf{Corollary 4}, $\gamma_b-$ dominating broadcast function value of $P_{n+1}$ is $\lceil \frac{n+1}{3} \rceil$. It is sufficient to find $\gamma_b-$ dominating broadcast function value of $S'$

Let $S''$  be the graph $S'$ with an external side $N'_1N'_2$. It can be observed that $S''$ is a graph $S^m_{n-1}$. By \textbf{Corollary 5} and $P(m-1)$,
\begin{align}
\gamma_b(S')\geq\gamma_b(S'')=\gamma_b(S^n_{m-1})=n+\lfloor \frac{m-1}{2}\rfloor\nonumber
\end{align}
Therefore,
\begin{align}
\gamma_b(S^n_m)&= \lceil \frac{n+1}{3} \rceil + \gamma_b(S') \nonumber\\
&\geq \lceil \frac{n+1}{3} \rceil + n+\lfloor \frac{m-1}{2}\rfloor\nonumber\\ 
&>n+\lfloor \frac{m}{2}\rfloor= rad(S_m^n) &&\text{($n>2$)} \nonumber
\end{align}
, which contradicts with\textbf{ Corollary 4} as desired. \end{proof}

However, if $n\leq 2$, the $\gamma_b-$dominating broadcast function value in \textbf{Claim 2} is equal to $rad(S_m^n)$. Therefore, it is sufficient to proof only in a case that there is no function in $B$ such that there is a broadcast vertex $M$ which can send a signal to the nearest base vertex $N$, and its signal strength is $n=d(M,N)$

\begin{claim} $A'\ne \emptyset$ \end{claim}

\begin{proof} We will prove this claim by the sake of contradiction. Assume that $A'$ is the empty set.

Given that $f$ is a function $B \rightarrow \mathbf{N} \cup \{0\}$, such that $f(x)$ is equal to the number of base vertices that receive signal from non-base vertex in the dominating broadcast $x \in B$. 

It is easy to see that $B$ is a finite set. From the extremal principle, there exists $x_0\in B$ such that $f(x_0)$ is the lowest value of $f(x)$. $(*)$

Because $A'$ is the empty set, $f(x)$ cannot be $0$, which means that $f(x_0)\ne 0$. Thus, there is a non-base vertex $M$ such that $M$ can send signal to base vertices. Let $N$ be a base vertex which is closest to the vertex $M$, and the distance between $N$ and $M$ is $k$ 

Assume that a signal strength at the vertex $M$ is $n$, and there are $y$ vertices on the branch containing M and N such that they do not lie between M and N.

\begin{figure} [htbp]
\centering
\includegraphics[width=10cm]{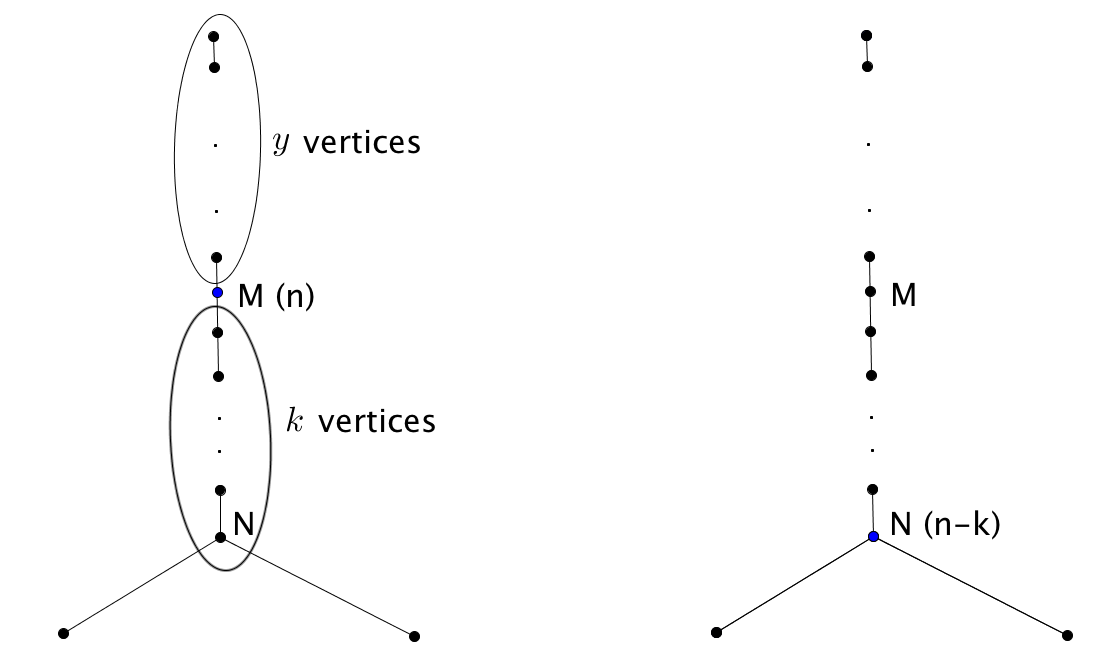}
\caption{The difference between the dominating broadcast function $x_0$ (left) and the dominating broadcast $x_1$ (right)}
\end{figure}

We will consider another dominating broadcast function $x'_0$ such that $x'_0$ and $x_0$ are the same, except vertices and their strengths on the branch containing $M$ and $N$. In the new broadcast function, the vertex $M$ is eliminated, and the vertex $N$ becomes a member of the broadcast function set with signal strength $n-k$ instead. Moreover, by \textbf{Definition 10}, pendant vertices that cannot receive signal from the vertex $N$ follow the standard pattern of ${\gamma}_{b}-$broadcast function.

If $y\geq n$, consider that a total of signal strengths in this branch for ${\gamma}_{b}-$ broadcast function $x_0$ is $n+\lceil \frac{y-n}{3}\rceil$, and a total of signal strengths in this branch for broadcast function $x'_0$ is $n-k+\lceil\frac{y-n+k}{3}\rceil$. 

\begin{center}
$n-k+\lceil\frac{y-n+k}{3}\rceil=n+\lceil\frac{y-n-2k}{3}\rceil \leq n+\lceil\frac{y-n}{3}\rceil$  
 
\end{center}

If $y<n$,consider that a total of signal strengths in this branch for ${\gamma}_{b}-$ broadcast function $x_0$ is $n$, and a total of signal strengths in this branch for the broadcast function $x'_0$ is $n-k+\lceil\frac{y-n+k}{3}\rceil$. 
\begin{center}
$n-k+\lceil\frac{y-n+k}{3}\rceil\leq n-k+\lceil\frac{k}{3}\rceil \leq n$  
 
\end{center}

Therefore, the cost of $x'_0$ is less than the cost of $x_0$, which means that $x'_0$ is also ${\gamma}_{b}-$broadcast function as $x_0$ is. Consider that the distance between $N$ and  the closest vertex which follows the standard pattern of $P_n$ is at least $2$. By \textbf{the significant property of the standard pattern}, $N$ cannot receive a signal from broadcast vertices in the path. Therefore, $N$ receives a signal only once for ${\gamma}_{b}-$broadcast function $x'_0$, $x'_0$ is a member of $B$.

However, $N$ in a new broadcast function cannot receive a signal from pendant vertices. Therefore, $f(x_0)-f(x'_0)=1$, which means that $f(x'_0)<f(x_0)$. This contradicts with the statement $(*)$, so we can conclude that $A'\ne \emptyset$ as desired. \end{proof}

Consider the proof for \textbf{Claim 3}, we assume the algorithm, which transforms the dominating broadcast function in the set $B$ to the dominating broadcast function in the set $B$ with the lower value of $f$, as $\Phi$.

If $\Phi$ is used for some branch, it does not affect to other branches because the power sent from that branch is still the same. Consider a branch where its vertices cannot send a signal to its base vertex. Therefore, if there are vertices on this branch such that they cannot receive a signal from a base vertex, after being used the algorithm $\Phi$, these vertices still cannot receive a signal from a base vertex. If there are vertices on this branch such that they can receive a signal from a base vertex, after being used the algorithm $\Phi$, these vertices still receive a signal from the same base vertex. Also, after being used the the algorithm $\Phi$, each vertex on this branch can still receive the same number of signals. We call those invariance, $V$.

\begin{figure} [htbp]
\centering
\includegraphics[width=8cm]{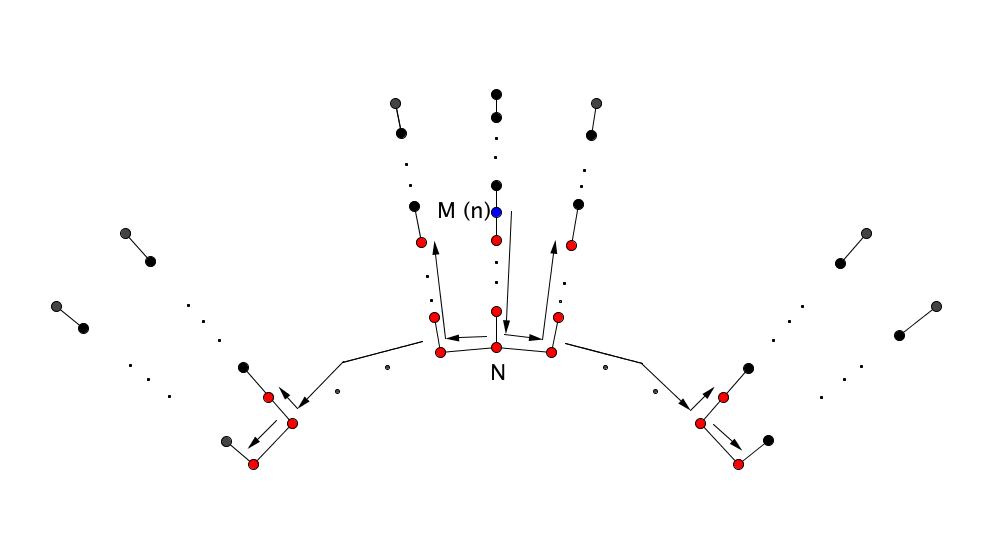}
\includegraphics[width=8cm]{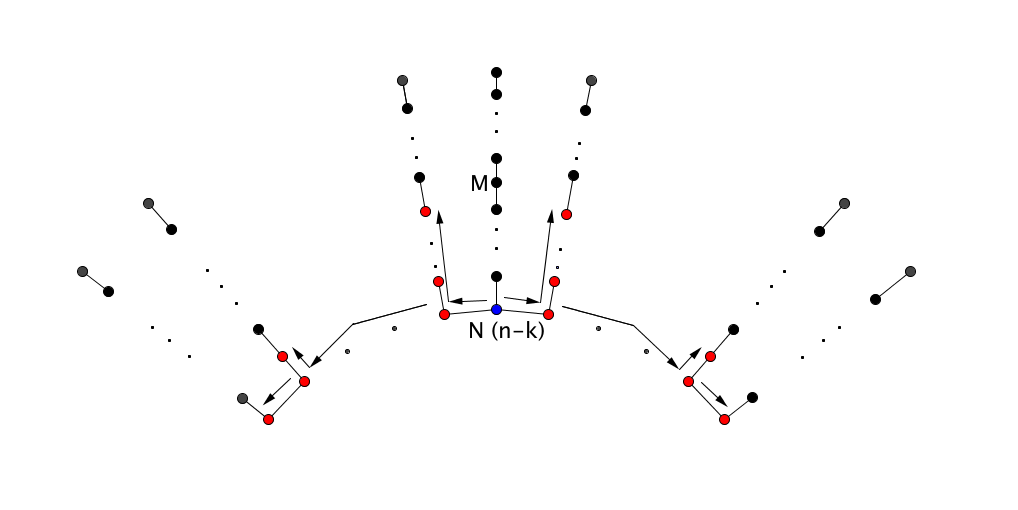}
\caption{(above) The graph before being used the algorithm $\Phi$ \newline (below) The graph after being used the algorithm $\Phi$ at the branch containing the vertex $N$ (a blue vertex) }
\end{figure}

Consider the efficient broadcast $x$. The algorithm $\Phi$ is used to transform $x$ for several times until the value of $f$ is $0$ to a new dominating broadcast $x'$. Thus, $x'\in A'$.

Consider a ${\gamma}_{b}-$broadcast function $x'$. Assume that $S$ is the set of base vertices that can send a signal. Let $|S|=t$ where $t$ is a natural number less than or equal $m$. Thus, we can order vertices in set S in clockwise order from $P_1,P_2,...,P_t$. For each $i \in \{1,2,...,t\}$, let the signal strength of $P_i$ be $p_i$.

We provide two cases for the natural number $t$

\textbf{Case $1$} $t>1$

Consider each $P_i$ as a figure shown. There are pendant vertices such that those vertices do not receive a signal from $P_i$, but their base vertices do. Assume that the set of these vertices are $T_i$. Therefore, from the definition of $A'$, those base vertices cannot receive a signal from $P_j$ where $j\in \{1,2,...,t\}-\{i\}$. Thus, those pendant vertices cannot receive a signal from any base vertices.

\begin{figure} [htbp]
\centering
\includegraphics[width=8cm]{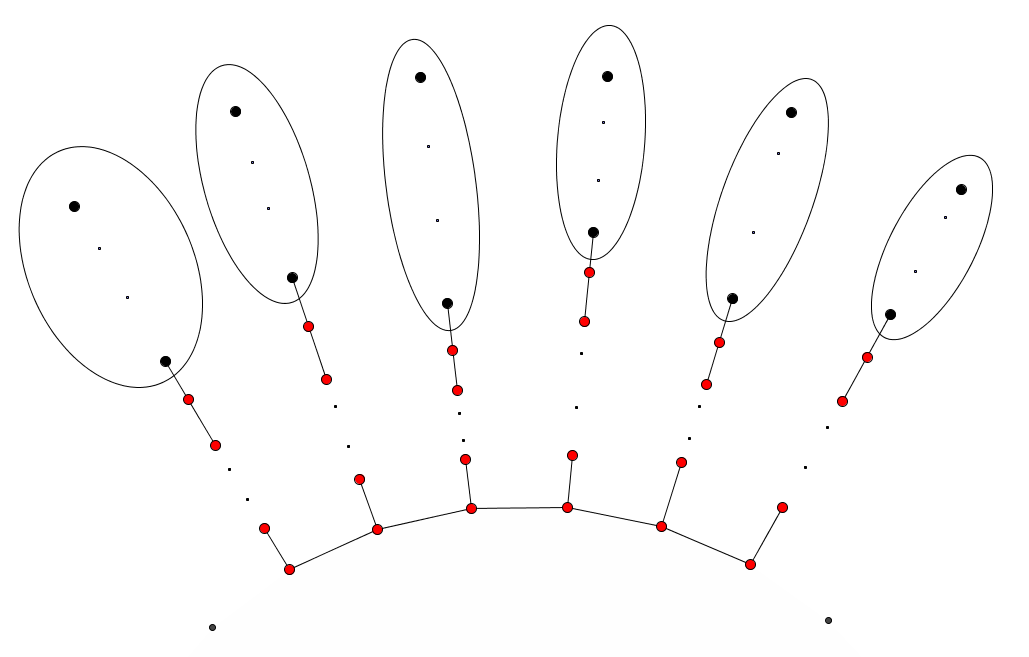}
\caption{Disconnected paths (in ovals) after a broadcast vertex sends a signal}
\end{figure}
 
Vertices in $T_i$ can be separated into several paths, corresponding to their branches. Thus, every two vertices in $T_i$ which comes from different branches cannot receive a signal from the same broadcast vertex. Otherwise, their base vertices would receive the signal twice. We can conclude that those several paths are independent to each other.

\begin{claim} $p_i<n$ for each $i \in \{1,2,...,t\}$ \end{claim}

\begin{proof} We will prove this claim by the sake of contradiction. Assume that there is $i_0 \in \{1,2,...,t\}$ such that $p_{i_0} > n$. By the definition of $A'$, $P_{i_0}$ cannot send a signal to all base vertices because there is another broadcast base vertex ($t>1$). Therefore, there is a branch that $P_{i_0}$ cannot send a signal to any vertex on the branch. Also, all vertices on the branch containing $P_{i_0}$ can receive the a signal from $P_{i_0}$. Consider neighborhood branches $A,B$, the number of vertices, which can receive a signal from $P_{i_0}$, of the branch $A$ differ at most one from those of the branch $B$. By the idea of the Intermediate Value Theorem, there is a branch $C$ such that only one vertex $T$ on the branch cannot receive a signal from $P_{i_0}$. It is obvious that $T$ is the leaf vertex on the branch $C$.

\begin{figure} [htbp]
\centering
\includegraphics[width=8cm]{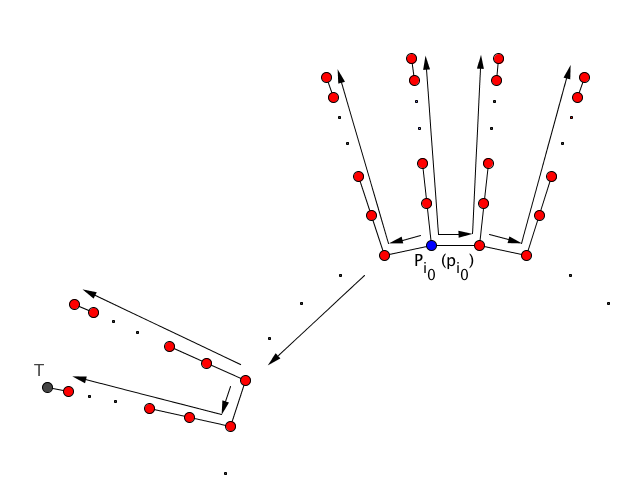}
\caption{The existence of the point $T$ after $P_{i_0}$ sends a signal with a strength $p_{i_0}$}
\end{figure}

Consider that the base vertex of the branch $C$ can receive a signal from $P_{i_0}$. Therefore, this branch is the branch where its vertices cannot send a signal to its base vertex. From the invariance $V$, $T$ cannot receive a signal from a base vertex, and vertices on the branch $C$ except $T$ cannot be broadcast vertices because they receive a signal from $P_{i_0}$. Also, those vertices still receive a signal only once due to the property of the efficient broadcast and the invariance $V$. Therefore, $T$ has to be a broadcast vertex. However, this broadcast vertex can send a signal to the neighbor vertex which interferes with the signal from $P_{i_0}$, which contradicts with the fact that this vertex still receive a signal only once, as desired. 
\end{proof}

Lengths of those disconnected paths in a clockwise order can form a tuple $(n,n-1,...,n-p_i+1,n-p_i,n-p_i+1,...,n-1,n)$. Therefore, from the standard pattern of ${\gamma}_{b}-$broadcast function, these disconnected paths and vertices receiving a signal from $P_i$ use costs at least   

\begin{center}
$p_i+\lceil \frac{n-p_i}{3} \rceil+2(\lceil\frac{n-p_i+1}{3}\rceil+\lceil\frac{n-p_i+2}{3}\rceil+...+\lceil\frac{n}{3}\rceil)$ 
\end{center}

Next, we will find the relationship among $p_1,p_2,...,p_t.$ Each vertex $P_i$ can send a signal to $2p_i+1$ vertices on the cycle, including itself. From the definition of $A'$, each vertex on the cycle can receive only once from another vertex on the cycle or itself. Therefore, 

\begin{equation}
(2p_1+1)+(2p_2+2)+...+(2p_k+1)=m
\end{equation}

Thus, we can conclude that

\begin{center}
$\gamma_b(S^n_m)=min\bigg($$\sum_{i=1}^k \Big(p_i+\lceil \frac{n-p_i}{3} \rceil+2(\lceil\frac{n-p_i+1}{3}\rceil+\lceil\frac{n-p_i+2}{3}\rceil+...+\lceil\frac{n}{3}\rceil)$$\Big)\bigg)$
\end{center}

Let $F(p_1,p_2,...,p_k)= \sum_{i=1}^k \Big(\lceil \frac{n-p_i}{3} \rceil+2(\lceil\frac{n-p_i+1}{3}\rceil+\lceil\frac{n-p_i+2}{3}\rceil+...+\lceil\frac{n}{3}\rceil)$\Big), for all positive number $a_1,a_2,...,a_k,k$ such that $(2p_1+1)+(2p_2+1)+...+(2p_k+1)=m$ 

Consider that $\gamma_b(S^n_m) = min (F(p_1,p_2,...,p_k) + (p_1+p_2+...+p_k)) = min(F(p_1,p_2,...,p_k)+ \frac{m-k}{2}).$ Therefore, it is sufficient to find the minimum value of function $F$

\begin{claim} If $p_1>p_2$, then $f(p_1,p_2,...,p_k)\leq f(p_1-1,p_2+1,p_3,...,p_k)$ \end{claim}

\begin{proof} Consider that $(2(p_1-1)+1)+(2(p_2+1)+1)+2p_3+...+2p_k=(2p_1+1)+(2p_2+1)+...+(2p_k+1)=m$, which means that $(p_1-1,p_2+1,p_3,...,p_k)$ still follows the condition $(1)$

\begin{align}
F(p_1,p_2,...)-F(p_1-1,p_2+1,...) & = \big(\lceil \frac{n-p_1}{3} \rceil +2\sum_{i=1}^{p_1} \lceil \frac{n-p_1+i}{3} \rceil + \lceil \frac{n-p_2}{3} \rceil\nonumber\\
& \quad +2\sum_{i=1}^{p_2} \lceil \frac{n-p_2+i}{3} \rceil\big)- \big(\lceil \frac{n-p_1+1}{3} \rceil\nonumber\\
& \quad +2\sum_{i=1}^{p_1-1} \lceil \frac{n-p_1+i+1}{3} \rceil + \lceil \frac{n-p_2-1}{3} \rceil\nonumber\\
& \quad +2\sum_{i=1}^{p_2+1} \lceil \frac{n-p_2+i-1}{3} \rceil\big) \nonumber\\
   & = \big(\lceil \frac{n-p_1}{3}\rceil-\lceil \frac{n-p_2-1}{3}\rceil \big)\nonumber\\
   & \quad +\big(\lceil \frac{n-p_1+1}{3}\rceil-\lceil \frac{n-p_2}{3}\rceil\big) \nonumber\\
   & \leq 0 \qquad(p_1\geq p_2+1) \nonumber
 \end{align}
   
Therefore, $f(p_1,p_2,...,p_k)\leq f(p_1-1,p_2+1,p_3,...,p_k)$ as desired.
\end{proof}
From \textbf{Claim 5},
\begin{equation}
f(1,p_1+p_2-1,p_3,...,p_n)\leq f(2,p_1+p_2-2,p_3,...,p_n) \leq f(p_1,p_2,p_3,...,p_n) \nonumber
\end{equation}

Thus, we can conclude that $f(1,p_1+p_2-1,p_3,...,p_n)\leq f(p_1,p_2,p_3,...,p_n)$ if $p_1\geq p_2$. By obvious mathematical induction, $f(p_1,p_2,p_3,...,p_n) \geq f(\underbrace{1,1,...1}_{k-1 \: 1's},\frac{m-3k+2}{2})$, which means that this function has a similar property with a convex function, the maximum value appearing at the extremal element in the domain.

Therefore, 
\begin{align}
\gamma_b(S^n_m) & = min (f(\underbrace{1,1,...1}_{k-1 \: 1's},\frac{m-3k+2}{2})+(p_1+p_2+...+p_k)) \nonumber\\
& = min (f(\underbrace{1,1,...1}_{k-1 \: 1's},\frac{m-3k+2}{2})+\frac{m-k}{2}) \nonumber
 \end{align}
 
Given that $G(k)=f(\underbrace{1,1,...1}_{k-1 \: 1's},\frac{m-3k+2}{2})+\frac{m-k}{2}$ for all natural number $k$.
\begin{align}
G(k)= (k-1)(\lceil\frac{n-1}{3}\rceil+2\lceil\frac{n}{3}\rceil)+\lceil\frac{n-x}{3}\rceil+2\sum_{i=1}^{x} \lceil \frac{n-x+i}{3} \rceil + \frac{m-k}{2} \nonumber
\end{align}
, where $x= \frac{m-3k+1}{2}$
\begin{claim} $G(k+2)\geq G(k)$ for each natural number $k$. \end{claim}

\begin{proof} Consider that 

\begin{align}
G(k+2)= (k+1)(\lceil\frac{n-1}{3}\rceil+2\lceil\frac{n}{3}\rceil)+\lceil\frac{n-x+3}{3}\rceil+2\sum_{i=1}^{x-3} \lceil \frac{n-x+i+3}{3} \rceil + \frac{m-k}{2}-1 \nonumber
\end{align}

Therefore,

\begin{align}
G(k+2)-G(k)& = 2(\lceil\frac{n-1}{3}\rceil+2\lceil\frac{n}{3}\rceil)-\lceil\frac{n-x}{3}\rceil-2\lceil \frac{n-x+1}{3} \rceil \nonumber\\
& \qquad - 2\lceil \frac{n-x+2}{3}\rceil-\lceil \frac{n-x+3}{3}\rceil-1 \nonumber\\
& = 2(\lceil\frac{n-1}{3}\rceil-\lceil \frac{n-x+2}{3}\rceil)+(\lceil\frac{n}{3}\rceil-\lceil\frac{n-x}{3}\rceil) \nonumber\\
& \qquad +2(\lceil\frac{n}{3}\rceil-\lceil \frac{n-x+1}{3} \rceil)+(\lceil\frac{n}{3}\rceil-\lceil \frac{n-x+3}{3}\rceil-1) \nonumber\\
& \geq \lceil\frac{n}{3}\rceil - \lceil\frac{n-x}{3}\rceil -1 \qquad (\text{That is because $x \geq 3$}) \nonumber
\end{align}

Because of $x \geq 3$, $G(k+2)-G(k) \geq \lceil\frac{n}{3}\rceil - \lceil\frac{n-3}{3}\rceil -1=0$ as desired.

\end{proof}

From \textbf{Claim 6},

\begin{align}
G(1)\leq G(3)\leq G(5)\leq ... \nonumber\\
G(2)\leq G(4)\leq G(6)\leq ... \nonumber
\end{align}

As a result, the minimum value of the function $G$ is either $G(1)$ or $G(2)$. $G(2)$ can be occurred if and only if $m$ is even and $m \geq 6$ (by  $x>0$)

\begin{align}
G(2) & = f(1,\frac{m}{2}-2) \nonumber\\
& = (\lceil\frac{n-1}{3}\rceil+2\lceil\frac{n}{3}\rceil)+\lceil\frac{n-x}{3}\rceil+2\sum_{i=1}^{x} \lceil \frac{n-x+i}{3} \rceil + \frac{m-2}{2} \nonumber\\
& \geq (\lceil\frac{n-1}{3}\rceil+2\lceil\frac{n}{3}\rceil) + 2\lceil\frac{n}{3}\rceil+ \frac{m-2}{2} \nonumber\\
& \geq (\lceil\frac{n-1}{3}\rceil+2\lceil\frac{n}{3}\rceil) + \frac{m}{2}+1 \nonumber\\
& > n+\frac{m}{2} \nonumber 
\end{align}

However, $rad(S_m^{n}) = n+ \frac{m}{2}$, which means that  $rad(S_m^{n})<G(2).$ Thus, $G(2)$ cannot be the value of $\gamma_b(S^n_m)$. Therefore, this case cannot be the $\gamma_b-$dominating broadcast function.

\textbf{Case 2} $t=1$

From \textbf{Case 1}, the minimum value of the function $G$ is $G(1)$, which is the condition of this case. There is only one broadcast base vertex, $P^*$, being able to send a signal to cover every base vertices.

Let the signal strength of $P^*$ be $rad(S_m^{n})-k$. If $k$ is more than $0$, there are $k$ vertices on the branch, $l^*$, opposite to $P^*$ that cannot receive the signal from $P^*$. Also, there are $k-1$ vertices on each branch which is neighbor with the branch  $l^*$. Therefore, The total cost of this broadcast function is

\begin{align}
rad(S^n_m)-k+\lceil \frac{k}{3} \rceil +\lceil \frac{k-1}{3} \rceil+\lceil \frac{k-1}{3} \rceil & \geq \gamma_b(S^n_m)-k+\lceil \frac{k+(k-1)+(k-1)}{3} \rceil \nonumber\\
& = rad(S^n_m)-k+\lceil \frac{3k-2}{3} \rceil \nonumber\\
& = rad(S^n_m)-k+k \nonumber\\
& = rad(S^n_m) \nonumber
\end{align}

Therefore, in the case of $k>0$, the $\gamma_b-$dominating function of $S_m^{n}$ cannot be less than the radius of the sunlet graph. In the case of $k=0$, the minimum cost of this function is also the radius of the sunlet graph. Finally, we can conclude that
\begin{align}
\gamma_b(S^n_m)=rad(S^n_m)=n+\lfloor \frac{m}{2}\rfloor \nonumber
\end{align} \end{proof}
\begin{figure} [htbp]
\centering
\includegraphics[width=6cm]{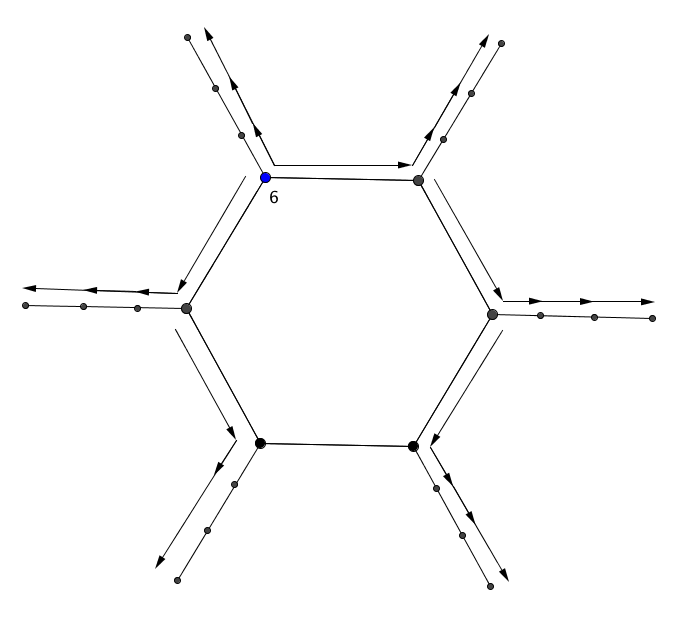}
\caption{The example of $\gamma_b-$ dominating broadcast function of $S^3_6$ ($m$ is an even number)}
\end{figure}

\begin{figure} [htbp]
\centering
\includegraphics[width=6cm]{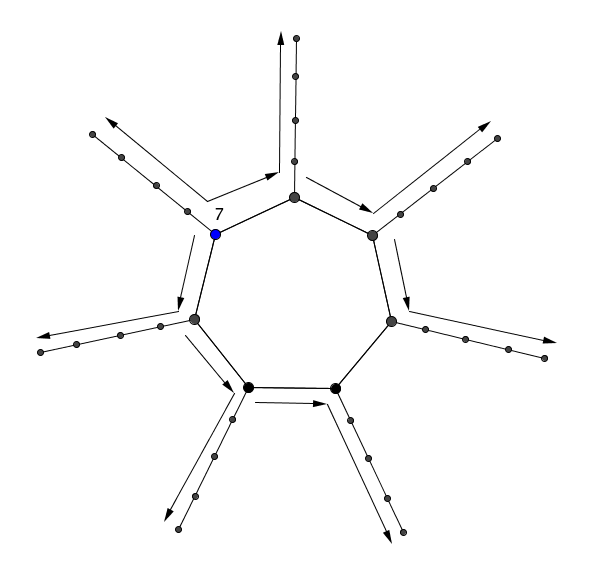}
\caption{The example of $\gamma_b-$ dominating broadcast function of $S^4_7$ ($m$ is an odd number)}
\end{figure}

\section{Discussion}
In this paper, the $\gamma_b-$dominating broadcast function values for the cycle graph and the $m-$Sunlet graph are evaluated. There is another variation of the cycle and the Sunlet, called \textbf{generalized Sunlet graph}. The \textbf{generalized Sunlet graph} is a cycle graph and straight paths extended from some vertices of a cycle graph. However, lengths of paths are not necessarily the same, which makes the proof more complicated. For our further work, we will evaluate the $\gamma_b-$dominating broadcast function values of the generalized Sunlet graph.

\begin{figure} [htbp]
\centering
\includegraphics[width=6cm]{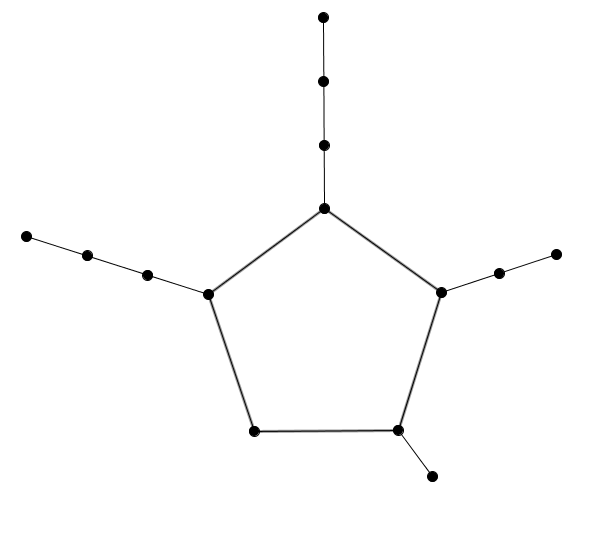}
\caption{The example of the generalized Sunlet graph}
\end{figure}
\bibliographystyle{plain}
\bibliography{bibliography}

\begin{thebibliography}{1}

\bibitem{anitha2008n}
R~Anitha and RS~Lekshmi.
\newblock N-sun decomposition of complete, complete bipartite and some harary
  graphs.
\newblock {\em International Journal of Computational and Mathematical
  Sciences}, 2(1):33--38, 2008.

\bibitem{dunbar2006broadcasts}
Jean~E Dunbar, David~J Erwin, Teresa~W Haynes, Sandra~M Hedetniemi, and
  Stephen~T Hedetniemi.
\newblock Broadcasts in graphs.
\newblock {\em Discrete Applied Mathematics}, 154(1):59--75, 2006.

\bibitem{erwin2001cost}
David~John Erwin.
\newblock Cost domination in graphs.
\newblock 2001.

\bibitem{haynes1998fundamentals}
Teresa~W Haynes, Stephen Hedetniemi, and Peter Slater.
\newblock {\em Fundamentals of domination in graphs}.
\newblock CRC Press, 1998.

\bibitem{haynes1998domination}
TW~Haynes, ST~Hedetniemi, and PJ~Slater.
\newblock Domination in graphs: the theory, 1998.

\end{thebibliography}

\end{document}